\documentclass{amsart}

\usepackage{amsmath,amsthm,amssymb,cite,setspace,fancyhdr,enumerate,xcolor}
\usepackage{xpatch}
\xapptocmd\normalsize{%
 \abovedisplayskip=12pt plus 3pt minus 9pt
 \abovedisplayshortskip=0pt plus 3pt
 \belowdisplayskip=12pt plus 3pt minus 9pt
 \belowdisplayshortskip=7pt plus 3pt minus 4pt
}{}{}
\theoremstyle{definition}
\newtheorem{definition}{Definition}[section]

\theoremstyle{plain}
\newtheorem{theorem}[definition]{Theorem}
\newtheorem{corollary}[definition]{Corollary}
\newtheorem{lemma}[definition]{Lemma}
\numberwithin{equation}{section}

\begin{document}

\title[Lemniscate Convexity of Generalized Bessel   Functions]{Lemniscate Convexity and Other Properties\\ of Generalized Bessel   Functions}

\author[V. Madaan]{Vibha Madaan}
\address{Department of Mathematics, University of Delhi, Delhi--110 007, India}
\email{vibhamadaan47@gmail.com}

\author[A. Kumar]{Ajay Kumar}
\address{Department of Mathematics, University of Delhi, Delhi--110 007, India}
\email{akumar@maths.du.ac.in}

\author[V. Ravichandran]{V. Ravichandran}
\address{Department of Mathematics, National Institute of Technology, Tiruchirappalli--620015, India}
\email{vravi68@gmail.com, ravic@nitt.edu}


\thanks{The first author is supported by University Grants Commission(UGC), UGC-Ref. No.:1069/(CSIR-UGC NET DEC, 2016).}

	\maketitle
	\begin{abstract}
		Sufficient conditions on associated parameters $p,b$ and $c$ are obtained so that the generalized and \textquotedblleft{normalized}\textquotedblright{} Bessel function $u_p(z)=u_{p,b,c}(z)$ satisfies $|(1+(zu''_p(z)/u'_p(z)))^2-1|<1$ or $|((zu_p(z))'/u_p(z))^2-1|<1$. We also determine the condition on these parameters so that $-(4(p+(b+1)/2)/c)u'_p(z)\prec\sqrt{1+z}$. Relations between the parameters $\mu$ and $p$ are obtained such that the normalized Lommel function of first kind  $h_{\mu,p}(z)$  satisfies the subordination $1+(zh''_{\mu,p}(z)/h'_{\mu,p}(z))\prec\sqrt{1+z}$. Moreover, the properties of Alexander transform of the function $h_{\mu,p}(z) $ are discussed.
	\end{abstract}
	\section{Introduction}
	We consider analytic functions $f$ defined on the unit disk $\mathbb{D}=\{z:|z|<1\},$  and  normalized by $f(0)=0=f'(0)-1$. The class of all these functions is denoted by $\mathcal{A}$ and its subclass  consisting of univalent ($\equiv$ one-to-one) functions is denoted by $\mathcal{S}$. For two analytic functions $f$ and $g$ on $\mathbb{D},\ f$ is said to be \textit{subordinate} to $g,$ written as $f(z)\prec g(z)$ (or $f\prec g$), if there is an analytic function $w:\mathbb{D}\to \mathbb{D}$ with $w(0)=0$ satisfying $f =g\circ w$. If $g$ is a univalent function, then $f(z)\prec g(z)$ if and only if $f(0)=g(0)$ and $f(\mathbb{D})\subset g(\mathbb{D})$. It is worth to mention that this concept of subordination is a natural generalization of inequalities to complex plane.  The class of convex functions (respectively starlike functions) consists of all those functions $f\in\mathcal{A}$ for which $f(\mathbb{D})$ is convex (respectively starlike with respect to origin) and is denoted by $\mathcal{K}$ (respectively ${\mathcal{S^*}}$). An analytic description of class $\mathcal{K}$ and $\mathcal{S^*}$ is as follows:
		\[
		 \mathcal{K}:=\left\{ f\in\mathcal{A} : \operatorname{Re}\left(1+\frac{zf''(z)}{f'(z)}\right)>0,\ z\in\mathbb{D}  \right\}\quad \text{and} \quad
		 \mathcal{S^*}:=\left\{  f\in\mathcal{A}:\operatorname {Re}\frac{zf'(z)}{f(z)}>0,\ z\in\mathbb{D}  \right\}.
		 \]
		 A function $f\in\mathcal{A}$ is \emph{lemniscate convex} if $1+(zf''(z)/f'(z))$ lies in the region bounded by right half of lemniscate of Bernoulli given by $\{ w: |w^2-1|=1  \}$. In terms of subordination, the function $f$ is called \emph{lemniscate convex} if $1+(zf''(z)/f'(z))\prec\sqrt{1+z}$ and similarly, the function $f$ is \emph{lemniscate starlike} if $zf'(z)/f(z)\prec\sqrt{1+z}$. On the other hand, the function $f\in\mathcal{A}$ is \emph{lemniscate Carath\'{e}odory} if $f'(z)\prec\sqrt{1+z}$. Since $\operatorname{Re}\sqrt{1+z}>0,$ a lemniscate Carath\'{e}odory function is a Carath\'{e}odory function and hence is univalent. Note that a lemniscate convex function $f$ satisfies \[ \left|\operatorname{arg}\left(1+\frac{zf''(z)}{f'(z)}\right)\right|<\frac{\pi}{4}  \] and hence it is strongly convex of order $1/2$. 
	
		
		The function $w(z):=w_{p,b,c}(z)$ is a particular solution of the second order linear differential equation
		\begin{equation}\label{bessel eqn}
		z^2w''(z)+bzw'(z)+(cz^2-p^2+p(1-b))w(z)=0, \quad{} \quad{} p,b,c\in\mathbb{C}
		\end{equation}
		and is given by
		\begin{equation}\label{gen bessel func}
		w_{p,b,c}(z)=\underset{n\geq0}{\sum}\frac{(-c)^n}{n!\,\Gamma(p+n+(b+1)/2)}\left(\frac{z}{2}\right)^{2n+p},\quad{} \quad{} z\in\mathbb{C}.
		\end{equation}
		The function $w_{p,b,c}(z)$ is called the generalized Bessel function of first kind of order $p$. For some particular values of $b$ and $c,$ the equation \eqref{bessel eqn} reduces to Bessel $(b=1,\,c=1),$ modified Bessel $(b=1,\,c=-1),$ spherical Bessel $(b=2,\,c=1),$ modified spherical Bessel $(b=2,\,c=-1)$ differential equations.
		
		To study the geometric properties such as univalence, starlikeness and convexity of Bessel function, modified Bessel function, spherical Bessel function and modified spherical Bessel function of first kind of order $p,$  we consider the normalization of $w_{p,b,c}(z)$ which is defined by the transfomation $u_{p,b,c}(z)=2^p\,\Gamma(p+(b+1)/2)z^{-p/2}w_{p,b,c}(\sqrt{z})$. Let $(\lambda)_n$ denote the Pochhammer (or Appell) symbol defined in terms of the Euler gamma function by
		\[ (\lambda)_n=\frac{\Gamma(\lambda+n)}{\Gamma(\lambda)}=\lambda(\lambda+1)\ldots(\lambda+n-1)  \]
		and $(\lambda)_0=1$. Using the Pochhammer symbol, the expression $u_{p,b,c}(z)$ becomes
		\begin{equation}\label{norm gen bessel func}
		u_p(z)=u_{p,b,c}(z)=\underset{n\geq 0}{\sum}\frac{(-c/4)^n}{(\kappa)_n}\frac{z^n}{n!},\quad{} \quad{} z\in\mathbb{C}
		\end{equation}
		where $p,b,c\in\mathbb{C}$ and $\kappa=p+(b+1)/2\neq 0,-1,-2,\ldots$ The function $u_{p,b,c}(z)$ is called generalized and \textquotedblleft {normalized}\textquotedblright{} Bessel function of first kind of order $p$. Let $b_n=(-c/4)^n/((\kappa)_n n!)$. Note that the function $u_p(z)$ is not normalized according to the usual definition of normalization but $(u_p-b_0)/b_1$ is. Therefore, the word \emph{normalized} has been put in quotes. Also, the series given in \eqref{norm gen bessel func} is convergent in the whole complex plane and hence the function $u_p(z)$ is an entire function. Note that the function $u_p(z)$ satisfies the differential equation \[4z^2u''_p(z)+4\kappa zu'_p(z)+czu_p(z)=0,\quad{} \quad{} z\in\mathbb{C}.  \]
		For a detailed study about the Bessel functions, one may refer \cite{MR2656410,MR2594845,MR3182021,MR3004986,MR3021327,MR3302161}. Let $\mathcal{S}(\alpha,\beta,\lambda)$ be a subclass of $\mathcal{A}$ satisfying $z/f(z)\neq0$ and
		\[ \left|f'(z)\left(\frac{z}{f(z)}\right)^2-\beta z^3\left(\frac{z}{f(z)}\right)'''-(\alpha+\beta)z^2\left(\frac{z}{f(z)}\right)''-1\right|\leq \lambda,\quad{} \quad{} z\in\mathbb{D}.  \] Baricz \emph{et al.} \cite{MR3352679} obtained sufficient conditions on the constants $\alpha>-1$ and $\beta$ such that the function $z/u_{p,b,c}(z)\in\mathcal{S}(\alpha,\beta,\lambda)$. Prajapat \cite{MR2826152} determined conditions for generalized bessel function (with a different normalization that one considered in this paper) to be univalent in the open unit disk. Kanas \emph{et al.} \cite{MR3716218} used the method of differential subordination to obtain sufficient conditions which imply that the function $u_{p,b,c}(z)$ is Janowski convex and $zu'_{p,b,c}(z)$ is Janowski starlike. The method of differential subordination was formulated by Miller and Mocanu \cite{MR1760285}. Radhika \emph{et al.} \cite{MR3753028} established sufficient conditions for Bessel function to be in class of Janowski starlike and Janowski convex functions. In \cite{MR2743533}, Baricz determined the conditions which imply that the function $u_{p,b,c}(z)$ is convex and $zu_{p,b,c}(z)$ is starlike of order $1/2$ in $\mathbb{D}$. Bohra \emph{et al.} \cite{MR3738359} obtained the conditions so that the functions $u_{p,b,c}(z)$ and $zu_{p,b,c}(z)$ are strongly convex of order $1/2$ and strongly starlike of order $1/2$ respectively in $\mathbb{D}.$
		
\medskip
		The Lommel function of first kind $s_{\mu,p}$ is a particular solution of inhomogeneous Bessel differential equation \[ z^2w''(z)+zw'(z)+(z^2-p^2)w(z)=z^{\mu+1}.  \] The function $s_{\mu,p}(z)$ can be expressed in terms of hypergeometric function
		\begin{equation}\label{lommel func}
		s_{\mu,p}(z)=\frac{z^{\mu+1}}{(\mu-p+1)(\mu+p+1)} {}_1F_2\left(1;\frac{\mu-p+3}{2},\frac{\mu+p+3}{2};\frac{-z^2}{4}\right),
		\end{equation}
		where $\mu\pm p$ is not a negative odd integer. Note that the Lommel function $s_{\mu,p}$ does not belong to the class $\mathcal{A}$. Hence, we consider the following normalization of the Lommel function of the first kind:
		\begin{equation}
		h_{\mu,p}(z)=(\mu-p+1)(\mu+p+1)z^{\frac{1-\mu}{2}}s_{\mu,p}(\sqrt{z}) \label{norm lommel func}=z+\underset{n=1}{\overset{\infty}{\sum}}\frac{(-1/4)^n}{(K)_n(F)_n}z^{n+1},\notag
		\end{equation}
		where $K=(\mu-p+3)/2$ and $F=(\mu+p+3)/2$. Clearly the function $h_{\mu,p}\in\mathcal{A}$ and satisfies the differential equation
		\begin{equation}\label{de for h}
		z^2 h''_{\mu,p}(z)+\mu zh'_{\mu,p}(z)+\left(\frac{(\mu+1)^2-p^2}{4}+\frac{z}{4}\right)h_{\mu,p}(z)=(\mu+1-p)(\mu+1+p)\frac{z}{4}.
		\end{equation}
		Ya\u{g}mur \cite{MR3353311} obtained conditions on the parameters  $\mu$ and $p$ such that the function satisfies $\operatorname{Re}(h_{\mu,p}(z)/z)>\alpha$ for $0\leq\alpha<1$. Baricz \emph{et al.} \cite{MR3596935} studied the zeroes of some normalization of Lommel and Struve function and hence determined the radius of convexity of these functions.
		
\medskip
In this paper, sufficient conditions on parameters $p,b$ and $c$ (and $\mu,p$) are derived so that the generalized and \textquotedblleft {normalized}\textquotedblright{} Bessel function $u_{p,b,c}(z)$ (and the normalization $h_{\mu,p}(z)$ of Lommel function $s_{\mu,p}(z)$) is lemniscate convex in $\mathbb{D}$. As an application of lemniscate convexity of generalized and \textquotedblleft {normalized}\textquotedblright{} Bessel function, a relation between the parameters $p,b$ and $c$ is obtained such that the function $zu_{p,b,c}(z)$ becomes lemniscate starlike in $\mathbb{D}$. Moreover, sufficient conditions are obtained on $p,b$ and $c$ for which the function $(-4\kappa/c)u_{p,b,c}(z)$ is lemniscate Carath\'{e}odory in $\mathbb{D},$ hence becoming close-to-convex and therefore univalent. Also, relations between the constants $\mu$ and $p$ are obtained that implies the Alexander transform of the function $h_{\mu,p}(z)$ is lemniscate convex and lemniscate Carath\'{e}odory in $\mathbb{D}$. The method of admissibility conditions for differential subordination formulated by Miller and Mocanu \cite{MR1760285} has been used to prove the stated results.
		
		\section{Main Results}
	
The following theorem describes the conditions on $\kappa$ and $c$ such that $(-4\kappa/c)u'_p(z)\prec\sqrt{1+z}.$
		\begin{theorem}\label{thm univalence of u}
			Let $\kappa,c\in\mathbb{C}$ be such that $c\neq0$ and satisfy
			\begin{equation}\label{eq univalence of u}
			\operatorname{Re}\kappa>\max\{0,|c|-3/4\},
			\end{equation} then $(-4\kappa/c)u'_p(z)\prec\sqrt{1+z}.$
		\end{theorem}

The next result gives sufficient conditions on the parameters $\kappa$ and $c$ so that the generalized and \textquotedblleft{normalized}\textquotedblright{} Bessel function $u_p$ is lemniscate convex in $\mathbb{D}.$
\begin{theorem}\label{thm convexity of u}
	If $b,p,c\in\mathbb{C}$ are such that $c\neq0$ and
	\begin{equation} \label{convex}
	\sqrt{3}|\kappa-2|+\frac{|c|}{4}<\sqrt{\frac{9}{8}+\frac{1}{\sqrt{2}}},
	\end{equation}
	then the function $u_p(z)$ is lemniscate convex in $\mathbb{D}.$
\end{theorem}

Baricz proved the recursive relation satisfied by $u_p(z)$ as given in
		\begin{lemma}\cite[Lemma 1.2, p.\,14]{MR2656410}\label{relation}
			If $b,p,c\in\mathbb{C}$ and $\kappa\neq0,-1,-2,\ldots,$ then the function $u_p(z)$ satisfies the relation $4\kappa u'_p(z)=-cu_{p+1}(z)$ for all $z\in\mathbb{C}.$
		\end{lemma}
If $\sqrt{3}|\kappa-3|+|c|/4<\sqrt{\frac{9}{8}+\frac{1}{\sqrt{2}}},$ then from Theroem \ref{thm convexity of u}, it follows that $1+(zu''_{p-1}(z)/u'_{p-1}(z))\prec\sqrt{1+z}$ and hence $(zu'_{p-1}(z))'/u'_{p-1}(z)\prec\sqrt{1+z}$ which means that $zu'_{p-1}(z)$ is lemniscate starlike in $\mathbb{D}$. Also Lemma \ref{relation} gives that $czu_p(z)=-4(\kappa-1)zu'_{p-1}(z)$. Therefore $zu_p(z)$ is lemniscate starlike in $\mathbb{D}$. Thus we have the following:
\begin{corollary}\label{lem starlikeness}
	If $b,p,c\in\mathbb{C}$ are such that
	\begin{equation}\label{starlikeness}
	\sqrt{3}|\kappa-3|+\frac{|c|}{4}<\sqrt{\frac{9}{8}+\frac{1}{\sqrt{2}}},
	\end{equation}
	then 
	the function $zu_p(z)$ is lemniscate starlike in $\mathbb{D}.$
\end{corollary}
For $b=1=c,$ the generalized Bessel function $w_p(z),$ as given in \eqref{gen bessel func}, reduces to the Bessel function of first kind of order $p,$ denoted by $J_p(z)$ is given by
\begin{equation}\label{Jp}
J_p(z)= \underset{n\geq0}{\sum}\frac{(-1)^n}{n!\,\Gamma(p+n+1)}\left(\frac{z}{2}\right)^{2n+p}.
\end{equation}
With $b=1,\,c=-1$ the function $w_p(z)$ reduces to the modified Bessel function of first kind of order $p,$ denoted by $I_p(z)$, is given by
\begin{equation}\label{Ip}
I_p(z)=\underset{n\geq0}{\sum}\frac{1}{n!\,\Gamma(p+n+1)}\left(\frac{z}{2}\right)^{2n+p}.
\end{equation}
With the values $b=1,\,c=1$ using Theorem \ref{thm convexity of u} and Corollary \ref{lem starlikeness}, we get the following
\begin{corollary}\label{jp2}
	Let $p\in\mathbb{C}$. For the function
	\[\mathcal{J}_p(z^{1/2})=2^p\Gamma(p+1)z^{-p/2}J_p(z^{1/2}),  \]
	where $J_p$ is the Bessel function as defined in \eqref{Jp}, the following holds:
	\begin{enumerate}[(i)]
		\item If $|p-1|\sqrt{3}<\sqrt{\frac{9}{8}+\frac{1}{\sqrt{2}}}-\frac{1}{4},$ then $\mathcal{J}_p(z^{1/2})$ is lemniscate convex in $\mathbb{D}$.
		\item If $|p-2|\sqrt{3}<\sqrt{\frac{9}{8}+\frac{1}{\sqrt{2}}}-\frac{1}{4},$ then $\mathcal{J}_p(z^{1/2})$ is lemniscate starlike in $\mathbb{D}$.
	\end{enumerate}
\end{corollary}
If $b=1,\,c=-1$ in Theorem \ref{thm convexity of u} and Corollary \ref{lem starlikeness}, then the function \[  \mathcal{I}_p(z^{1/2})=2^p\Gamma(p+1)z^{-p/2}I_p(z^{1/2}),  \] where $I_p$ is the modified Bessel function of the first kind of order $p,$ has the properties same as that for the function $\mathcal{J}_p(z^{1/2}),$ because $|c|=1$ in this case.

Since some of the Bessel functions of first kind of order $p$ can also be expressed in terms of trigonometric functions like cos, sin, cosh, and sinh, we have some relations for some trigonometric ratios to be lemniscate convex or lemniscate starlike in $\mathbb{D}$ which are as follows:

Since
\[
\mathcal{J}_{1/2}(z^{1/2})=\sqrt{\frac{\pi}{2}}z^{-1/4}J_{1/2}(\sqrt{z})=\frac{\sin\sqrt{z}}{\sqrt{z}}
\quad \text{and}\quad \mathcal{I}_{1/2}(z^{1/2})=\sqrt{\frac{\pi}{2}}z^{-1/4}I_{1/2}(\sqrt{z})=\frac{\sinh\sqrt{z}}{\sqrt{z}},
\]
the functions $(\sin\sqrt{z})/\sqrt{z}$ and $(\sinh\sqrt{z})/\sqrt{z}$ are lemniscate convex in $\mathbb{D}.$
Also, since \[ z\mathcal{J}_{3/2}(z^{1/2})=3\left(\frac{\sin\sqrt{z}}{\sqrt{z}}-\cos\sqrt{z}\right), \]
the function $(\sin\sqrt{z}-\sqrt{z}\cos\sqrt{z})/\sqrt{z}$ is lemniscate starlike in $\mathbb{D}.$

\medskip
	Now we obtain conditions on $\mu$ and $p$ such that the function $h_{\mu,p}(z)$ is lemniscate convex in $\mathbb{D}.$
\begin{theorem}\label{thm convexity of h}
	Let $\mu,p\in\mathbb{R}$ be such that $\mu\pm p$ is not an odd negative integer. If
	\begin{equation}\label{eq h convexity}
	\frac{3\mu}{2\sqrt{2}}-\sqrt{3}\left|\frac{(\mu+1)^2-p^2}{4}-2\mu-2\right|>\frac{13\sqrt{3}}{4}-\frac{15}{8\sqrt{2}}+\frac{1}{2},
	\end{equation}
	then the function $h_{\mu,p}(z)$ is lemniscate convex in $\mathbb{D}.$
\end{theorem}

As the function $h_{\mu,p}(z)$ is lemniscate convex in $\mathbb{D}$ for $\mu$ and $p$ satisfying \eqref{eq h convexity}, the function $zh'_{\mu,p}(z)$ is lemniscate starlike in $\mathbb{D}$ for $\mu,p$ as in \eqref{eq h convexity}.
	
	The convolution or Hadamard product of two functions $f$ and $g$ having power series expansion as $f(z)=\sum_{n=0}^{\infty}a_nz^n$ and $g(z)=\sum_{n=0}^{\infty}b_nz^n$ is defined as \[ (f*g)(z)=\underset{n=0}{\overset{\infty}{\sum}}a_nb_nz^n. \]
Shanmugam \cite{MR0994916} proved that the for a function $g\in\mathcal{A}$ and a convex univalent function $h$ satisfying $h(0)=1$ and $\operatorname{Re}h(z)>0,$ the class $K_g(h):=\left\{ f\in\mathcal{A}| (g*f)'(z)\neq0\ ,1+\dfrac{z(g*f)''(z)}{(g*f)'(z)}\prec h(z) \text{ for }z\in\mathbb{D} \right\}$ is closed with respect to convolution with convex functions. In particular, for $g(z)=z/(1-z)$ and $h(z)=\sqrt{1+z},$ the class of lemniscate convex functions is closed with respect to convolution with convex functions.

The Alexander operator $A:\mathcal{A}\to\mathcal{A}$ is defined by \[ A[f](z):=\underset{0}{\overset{z}{\int}}\frac{f(t)}{t}dt=-\log[1-z]*f(z).  \]
The Libera operator $L:\mathcal{A}\to\mathcal{A}$ is defined by \[ L[f](z):=\frac{2}{z}\underset{0}{\overset{z}{\int}}f(t)dt=\frac{-2(z+\log[1-z])}{z}*f(z). \]
Thus, by Theorem \ref{thm convexity of h}, we have the following
\begin{theorem}\label{convolution}
	If $\mu,p$ satisfy \eqref{eq h convexity}, then $(h_{\mu,p}*f)(z)$ is lemniscate convex in $\mathbb{D}$ and thus the functions $A[h_{\mu,p}](z)$ and $L[h_{\mu,p}](z)$ are lemniscate convex in $\mathbb{D}.$
\end{theorem}
Now, consider the Alexander transform of the function $h_{\mu,p}(z)$ named as $f_{\mu,p}:\mathbb{D}\to\mathbb{C}$ by \[ f_{\mu,p}(z):=\underset{0}{\overset{z}{\int}}\frac{h_{\mu,p}(t)}{t}dt. \]
The function $f_{\mu,p}(z)$ is analytic in $\mathbb{D}$. Moreover, $f_{\mu,p}\in\mathcal{A}$. As $h_{\mu,p}(z)$ satisfies the differential equation \eqref{de for h}, $f_{\mu,p}(z)$ satisfies the differential equation
\[ z^2f_{\mu,p}'''(z)+(\mu+2)zf_{\mu,p}''(z)+\left(\frac{(\mu+1)^2-p^2}{4}+\frac{z}{4}\right)f'_{\mu,p}(z)=\frac{(\mu+1)^2-p^2}{4}. \]
Differentiating and dividing by $f'_{\mu,p}(z)$ and multipying by $z,$ we get \[ \frac{z^3f_{\mu,p}^{(4)}(z)}{f_{\mu,p}'(z)}+(\mu+4)\frac{z^2f'''_{\mu,p}(z)}{f'_{\mu,p}(z)}+\left(\frac{(\mu+1)^2-p^2}{4}+\frac{z}{4}+\mu\right)\frac{zf''_{\mu,p}(z)}{f'_{\mu,p}(z)}+\frac{z}{4}=0.  \]

Using Theroem \ref{convolution}, the function $f_{\mu,p}(z)$ is lemniscate convex in $\mathbb{D}$ for $\mu,p$ satisfying \eqref{eq h convexity}. Furthermore, the next theorem admits the conditions so that the function $f_{\mu,p}(z)$ defined above is lemniscate convex in $\mathbb{D}.$
\begin{theorem}\label{thm convexity of f}
	Let $\mu,p\in\mathbb{R}$ such that $\mu\pm p$ is not a negative odd integer. If $\mu,p$ satisfy
	\begin{equation}\label{eq f convexity}
	\frac{3\mu}{2\sqrt{2}}-\sqrt{3}\left|\frac{(\mu+1)^2-p^2}{4}-2\mu-2\right|>\frac{13\sqrt{3}}{4}-\frac{15}{8\sqrt{2}}+\frac{1}{4},
	\end{equation}
	then the function $f_{\mu,p}(z)$ is lemniscate convex in $\mathbb{D}$.  
\end{theorem}
The following lemma gives conditions on the constants $\mu,p$ such that the function $f_{\mu,p}(z)$ is lemniscate Carath\'{e}odory.
\begin{theorem}\label{thm univalence of f}
	Let $\mu,p\in\mathbb{C}$ be such that $\mu\pm p$ is not negative odd integer and $\operatorname{Re}\mu>-1$ and satisfy
	\begin{equation}
	\left|\frac{(\mu+1)^2-p^2}{2} \right|\sqrt{3}<\frac{\operatorname{Re}\mu}{2\sqrt{2}}+\frac{3}{8\sqrt{2}},
	\end{equation}
	then the function $f'_{\mu,p}(z)=\dfrac{h_{\mu,p}(z)}{z}\prec\sqrt{1+z}.$
\end{theorem}

\section{Proof of the main results}

The proofs of our theorems are based on the theory of first and second order differential subordination and the following results are needed to prove our results.
		\begin{lemma}\label{adm lemma}\cite{MKR}
			Let $p\in \mathcal{H}[1,n]$ with $p(z)\not\equiv 1$ and $n\geq 1$. Let $\Omega\subset\mathbb{C}$ and $\psi:D\subset\mathbb{C}^3\times\mathbb{D}\to\mathbb{C}$ satisfy
			\[\psi(r,s,t;z)\not\in\Omega\ \text{ whenever } z\in\mathbb{D},  \]
			$r=\sqrt{2\cos 2\theta}e^{i\theta},\ s=me^{3i\theta}/(2\sqrt{2\cos2\theta})$ and $\operatorname{Re}((t+s)e^{-3i\theta})\geq 3m^2/(8\sqrt{2\cos2\theta})$ where $m\geq n\geq1$ and $-\pi/4<\theta<\pi/4$. If $(p(z),zp'(z),z^2p''(z);z)\in D$ for $z\in \mathbb{D}$ and $\psi(p(z),zp'(z),z^2p''(z);z)\in \Omega, \ z\in \mathbb{D},$ then $p(z)\prec\sqrt{1+z}.$
		\end{lemma}
		In the case $\psi:\mathbb{C}^2\times\mathbb{D}\to\mathbb{C},$ the condition in Lemma \ref{adm lemma} reduces to $\psi(r,s;z)\not\in\Omega$ whenever $r=\sqrt{2\cos 2\theta}e^{i\theta},\ s=me^{3i\theta}/(2\sqrt{2\cos2\theta})$ for $m\geq n\geq1,\ -\pi/4<\theta<\pi/4$ and $z\in\mathbb{D}.$
		
		\medskip
		Baricz in \cite{MR2656410} determined the following conditions on $\kappa$ and $c$ for which the function $u_p(z)$ is univalent in $\mathbb{D}.$
		\begin{lemma}\cite[Theorem 2.9, p.\,29]{MR2656410}\label{non zero}
			If $b,p,c\in\mathbb{C}$ are such that $\operatorname{Re}\kappa>|c|/4+1,$ then $\operatorname{Re}u_p(z)>0$ for all $z\in\mathbb{D}$. Further, if $\operatorname{Re}\kappa>|c|/4$ and $c\neq0,$ then $u_p$ is univalent in $\mathbb{D}.$		
		\end{lemma}

	\begin{proof}[Proof of Theorem~\ref{thm univalence of u}]
		Define the function $p:\mathbb{D}\to\mathbb{C}$ by \[p(z)=\frac{-4\kappa}{c}u'_p(z).   \]
		Then the function $p(z)$ is analytic in $\mathbb{D}$ and $p(0)=1$. As the function $u_p(z)$ satisfies the differential equation $4z^2u''_p(z)+4\kappa zu'_p(z)+czu_p(z)=0,$ the function $p(z)$ satisfies
		\[ 4z^2p''(z)+4(\kappa+1)zp'(z)+czp(z)=0.  \]
		Define the function $\psi:\mathbb{C}^3\times\mathbb{D}\to\mathbb{C}$ by $\psi(r,s,t;z)=4t+4(\kappa+1)s+czr$ and let $\Omega:=\{0\}$. Then, for all $z\in\mathbb{D}$ we have that $\psi(p(z),zp'(z),z^2p''(z);z)\in\Omega.$
		For $r,s,t$ as in Lemma \ref{adm lemma}, we have
		\begin{align*}
		\left|\frac{\psi(r,s,t;z)}{4}\right|&=\left|t+s+\kappa\frac{me^{3i\theta}}{2\sqrt{2\cos2\theta}}+\frac{c}{4}\sqrt{2\cos2\theta}e^{i\theta}z\right|\\
		&\geq \left|(t+s)e^{-3i\theta}+\kappa\frac{m}{2\sqrt{2\cos2\theta}}\right|-\frac{|c|}{4}|z|\sqrt{2\cos2\theta}\\
		&\geq \operatorname{Re}((t+s)e^{-3i\theta})+\frac{m\operatorname{Re}\kappa}{2\sqrt{2\cos2\theta}}-\frac{|c|}{4}\sqrt{2}.\\
		&\geq \frac{3m^2}{8\sqrt{2\cos2\theta}}+\frac{m\operatorname{Re}\kappa}{2\sqrt{2\cos2\theta}}-\frac{|c|}{4}\sqrt{2}.
		\intertext{By hypothesis, $\operatorname{Re}\kappa>0$ and since $m\geq1$ so we have}
		\left|\frac{\psi(r,s,t;z)}{4}\right|&\geq \frac{3}{8\sqrt{2}}+\frac{\operatorname{Re}\kappa}{2\sqrt{2}}-\frac{|c|}{4}\sqrt{2}.
		\end{align*}
		Therefore, we get $\psi(r,s,t;z)\neq0$ for $r=\sqrt{2\cos2\theta}e^{i\theta},\, s=me^{3i\theta}/(2\sqrt{2\cos2\theta}),\, \operatorname{Re}((t+s)e^{-3i\theta})\geq 3m^2/(8\sqrt{2\cos2\theta})$ for $m\geq n\geq1,\, -\pi/4<\theta<\pi/4$ and $z\in\mathbb{D}$ if $\operatorname{Re}\kappa>\max\{0,|c|-3/4\}$. Hence, by Lemma \ref{adm lemma}, the theorem follows.
	\end{proof}

	\begin{proof}[Proof of Theorem~\ref{thm convexity of u}]
		Define the function $p:\mathbb{D}\to\mathbb{C}$ by \[ p(z)=1+\frac{zu_p''(z)}{u_p'(z)}.  \]
		Since, for $\kappa,c\in\mathbb{C}$ satisfying \eqref{convex}, $\operatorname{Re}\kappa>|c|/4$ also holds, Lemma \ref{non zero} implies that $u_p$ is univalent in $\mathbb{D}$ and thus $u'_p(z)\neq0$ for all $z\in\mathbb{D}$. The function $p(z)$ therefore is analytic in $\mathbb{D}$ and $p(0)=1$. Since $u_p$ satisfies the equation $4z^2u''_p(z)+4\kappa zu'_p(z)+czu_p(z)=0,$ the function $p(z)$ satisfies the differential equation
		\[ 4(zp'(z)-(p(z)-1)+(p(z)-1)^2)+4(\kappa+1)(p(z)-1)+cz=0.  \]
		Let $\psi:\mathbb{C}^2\times\mathbb{D}\to\mathbb{C}$ be defined by
		\[ \psi(r,s;z)=4(s-(r-1)+(r-1)^2)+4(\kappa+1)(r-1)+cz  \] and let $\Omega:=\{0\}$. Then $\psi(p(z),zp'(z);z)\in\Omega$ for all $z\in\mathbb{D}.$
		For $r,s$ as in Lemma \ref{adm lemma}, we have
		\begin{align}
		\left|\frac{\psi(r,s;z)}{4}\right|&=\left|s+r^2-1+(\kappa-2)(r-1)+\frac{c}{4}z\right|.\nonumber\\
		&\geq |s+r^2-1|-|k-2||r-1|-\left|\frac{c}{4}z\right|.\label{expr}
		\end{align}
		We first note that, using Lemma \ref{adm lemma},
		\begin{align*}
		|s+r^2-1|^2&=\left|\frac{me^{3i\theta}}{2\sqrt{2\cos2\theta}}+e^{4i\theta}\right|^2=\left|\frac{m}{2\sqrt{2\cos2\theta}}+e^{i\theta}\right|^2\\
		&=\frac{m^2}{8\cos2\theta}+1+\frac{m\cos\theta}{\sqrt{2\cos2\theta}}\geq \frac{9}{8}+\frac{1}{\sqrt{2}}
		\end{align*}
		and $|r-1|^2=2\cos2\theta+1-2\sqrt{2\cos2\theta}\cos\theta \leq 3$. Using these in equation \eqref{expr}, we get \[\left|\frac{\psi(r,s;z)}{4}\right|\geq\sqrt{\frac{9}{8}+\frac{1}{\sqrt{2}}}-\sqrt{3}|k-2|-\frac{|c|}{4}.   \]
		Hence, we get $\psi(r,s;z)\neq0$ for $r=\sqrt{2\cos 2\theta}e^{i\theta},\ s=me^{3i\theta}/(2\sqrt{2\cos2\theta})$ for $m\geq n\geq1,\ -\pi/4<\theta<\pi/4$ and $z\in\mathbb{D}$ if $\sqrt{3}|\kappa-2|+|c|/4<\sqrt{\frac{9}{8}+\frac{1}{\sqrt{2}}}$ which holds by given hypothesis. By Lemma \ref{adm lemma}, the theorem follows.	
	\end{proof}
	
	\begin{proof}[Proof of Theorem~\ref{thm convexity of h}]
		Define the function $p:\mathbb{D}\to\mathbb{C}$ by \[ p(z)=1+\frac{zh''_{\mu,p}(z)}{h'_{\mu,p}(z)}. \]
		As in \cite[Theorem 2.1]{MR3353311}, \[ |h'_{\mu,p}(z)|>\frac{2MN-4M-3N}{N(2M-3)}\ \ (M>3/2) \] where $M=(\mu+5)^2-p^2$ and $N=(\mu+3)^2-p^2$. Since for $\mu,p$ satisfying \eqref{eq h convexity}, $M>3/2$ and \[ |h'_{\mu,p}(z)|>\frac{2MN-4M-3N}{N(2M-3)}>0,  \] the function $p(z)$ is analytic in $\mathbb{D}$ and $p(0)=1.$
		
		On differentiating the equation \eqref{de for h}, dividing by $h'_{\mu,p}(z)$ and multiplying by $z$, we get
		\begin{equation}\label{h 4th derivative}
		\begin{split}
		&z^3\frac{h^{(4)}_{\mu,p}(z)}{h'_{\mu,p}(z)}+(\mu+4)\frac{z^2h^{(3)}_{\mu,p}(z)}{h'_{\mu,p}(z)}\\
		&+\left(\mu+2+\frac{(\mu+1)^2-p^2}{4}+\frac{z}{4}\right)\frac{zh''_{\mu,p}(z)}{h'_{\mu,p}(z)}+\frac{z}{2}=0.
		\end{split}
		\end{equation}
		Using \eqref{h 4th derivative}, we see that the function $p(z)$ satisfies
		\begin{align*}
		&\quad{}z^2p''(z)-2(zp'(z)-(p(z)-1))-2(p(z)-1)^3+3(p(z)-1)\\
		&\quad{}(zp'(z)-(p(z)-1)+(p(z)-1)^2)+(\mu+4)(zp'(z)-(p(z)-1)+(p(z)-1)^2)\\
		&+\left(\mu+2+\frac{(\mu+1)^2-p^2}{4}+\frac{z}{4}\right)(p(z)-1)+\frac{z}{2}=0.
		\end{align*}
		Define $\psi:\mathbb{C}^3\times\mathbb{D}\to\mathbb{C}$ by
		\begin{align*}
		\psi(r,s,t;z)&=t-2(s-(r-1))-2(r-1)^3+3(r-1)\\
		&\quad{}\quad{}(s-(r-1)+(r-1)^2)+(\mu+4)(s-(r-1)+(r-1)^2)\\
		&\quad{}+\left(\mu+2+\frac{(\mu+1)^2-p^2}{4}+\frac{z}{4}\right)(r-1)+\frac{z}{2}\\
		&=t+s+3rs+(\mu-2)s+(\mu+1)(r^2-1)+(r-1)^3\\
		&\quad{}+(r-1)\left(\frac{(\mu+1)^2-p^2}{4}-2\mu-2+\frac{z}{4}\right)+\frac{z}{2}.
		\end{align*}
		Let $\Omega:=\{0\}$. Then we have that $\psi(p(z),zp'(z),z^2p''(z);z)\in\Omega$ for all $z\in\mathbb{D}$. 
		For $r,s,t$ as in Lemma \ref{adm lemma}, we have
		\begin{align*}
		|\psi(r,s,t;z)|&=\Bigg|t+s+3\frac{m}{2}e^{4i\theta}+(\mu-2)\frac{me^{3i\theta}}{2\sqrt{2\cos2\theta}}+(\mu+1)e^{4i\theta}\\
		&\quad{}+(\sqrt{2\cos2\theta}e^{i\theta}-1)^3+\left(\frac{(\mu+1)^2-p^2}{4}-2\mu-2+\frac{z}{4}\right)\\
		&\quad{}\quad{}(\sqrt{2\cos2\theta}e^{i\theta}-1)+\frac{z}{2}\Bigg|\\
		&\geq \left|(t+s)e^{-3i\theta}+3\frac{m}{2}e^{i\theta}+(\mu-2)\frac{m}{2\sqrt{2\cos2\theta}}+(\mu+1)e^{i\theta}\right|\\
		&\quad{}-|\sqrt{2\cos2\theta}e^{i\theta}-1|^3-\left|\frac{(\mu+1)^2-p^2}{4}-2\mu-2\right||\sqrt{2\cos2\theta}e^{i\theta}-1|\\
		&\quad{}-\frac{|\sqrt{2\cos2\theta}e^{i\theta}-1|}{4}|z|-\frac{|z|}{2}\\
		&\geq \operatorname{Re}(t+s)e^{-3i\theta}+\frac{3m}{2}\cos\theta+(\mu-2)\frac{m}{2\sqrt{2\cos2\theta}}+(\mu+1)\cos\theta\\
		&\quad{}-\frac{13\sqrt{3}}{4}-\left|\frac{(\mu+1)^2-p^2}{4}-2\mu-2\right|\sqrt{3}-\frac{1}{2}\\
		&\geq \frac{3m^2}{8\sqrt{2\cos2\theta}}+\frac{3m}{2\sqrt{2}}+(\mu-2)\frac{m}{2\sqrt{2\cos2\theta}}+(\mu+1)\frac{1}{\sqrt{2}}-\frac{13\sqrt{3}}{4}\\
		&\quad{}-\left|\frac{(\mu+1)^2-p^2}{4}-2\mu-2\right|\sqrt{3}-\frac{1}{2}.
		\intertext{For $\mu,p$ as in \eqref{eq h convexity}, $\mu>2,$ we have}
		|\psi(r,s,t;z)|&\geq \frac{15}{8\sqrt{2}}+\frac{3\mu}{2\sqrt{2}}-\left|\frac{(\mu+1)^2-p^2}{4}-2\mu-2\right|\sqrt{3}-\frac{13\sqrt{3}}{4}-\frac{1}{2}.
		\end{align*}
		Hence, we get $\psi(r,s,t;z)\neq0$ for $r=\sqrt{2\cos 2\theta}e^{i\theta},\ s=me^{3i\theta}/(2\sqrt{2\cos2\theta}),\text{ and } \operatorname{Re}((t+s)e^{-3i\theta})\geq 3m^2/(8\sqrt{2\cos2\theta})$ for $m\geq n\geq1,\ -\pi/4<\theta<\pi/4$ and $z\in\mathbb{D}$ if \[\frac{3\mu}{2\sqrt{2}}-\sqrt{3}\left|\frac{(\mu+1)-p^2}{4}-2\mu-2\right|>\frac{13\sqrt{3}}{4}-\frac{15}{8\sqrt{2}}+\frac{1}{2},\] which holds by given hypothesis. By Lemma \ref{adm lemma}, the theorem follows.
	\end{proof}

	\begin{proof}[Proof of Theorem~\ref{thm convexity of f}]
		Define the function $p:\mathbb{D}\to\mathbb{C}$ by \[ p(z)=1+\frac{zf''_{\mu,p}(z)}{f'_{\mu,p}(z)} . \]
		As $f'_{\mu,p}(z)=h_{\mu,p}(z)/z,$ using \cite[Corollary 2.4]{MR3353311}, $p(z)$ is analytic in $\mathbb{D}$ for $\mu,p$ satisfying \eqref{eq f convexity} and $p(0)=1$. Then the function $p(z)$ satisfies the differential equation
		\begin{align*}
		&z^2p''(z)-2(zp'(z)-(p(z)-1)+(p(z)-1)^2)+3(zp'(z)-(p(z)-1)+(p(z)-1)^2)\\
		&(p(z)-1)+2(p(z)-1)^2-2(p(z)-1)^3+(\mu+4)(zp'(z)-(p(z)-1)+(p(z)-1)^2)\\
		&+\left(\frac{(\mu+1)^2-p^2}{4}+\frac{z}{4}+\mu\right)(p(z)-1)+\frac{z}{4}=0.
		\end{align*}
		Define $\psi:\mathbb{C}^3\times\mathbb{D}\to\mathbb{C}$ by
		\begin{align*}
		\psi(r,s,t;z)&=t-2(s-(r-1)+(r-1)^2)+3(s-(r-1)+(r-1)^2)(r-1)\\
		&\quad{}+2(r-1)^2-2(r-1)^3+(\mu+4)(s-(r-1)+(r-1)^2)\\
		&\quad{}+\left(\frac{(\mu+1)^2-p^2}{4}+\frac{z}{4}+\mu\right)(r-1)+\frac{z}{4}\\
		&=t+s+3sr+(\mu-2)s+(\mu+1)(r^2-1)+(r-1)^3\\
		&\quad{}+(r-1)\left(\frac{(\mu+1)^2-p^2}{4}+\frac{z}{4}-2\mu-2\right)+\frac{z}{4}.
		\end{align*}
		Let $\Omega:=\{0\}$. Clearly $\psi(p(z),zp'(z),z^2p''(z);z)\in\Omega$ for all $z\in\mathbb{D}$. 
		For $r,s,t$ as in Lemma \ref{adm lemma}, we have
		\begin{align*}
		|\psi(r,s,t;z)|&=\Bigg|t+s+\frac{3m}{2}e^{4i\theta}+(\mu-2)\frac{me^{3i\theta}}{2\sqrt{2\cos2\theta}}+(\mu+1)e^{4i\theta}+(\sqrt{2\cos2\theta}e^{i\theta}-1)^3\\
		&\quad{}+\left(\frac{(\mu+1)^2-p^2}{4}+\frac{z}{4}-2\mu-2\right)(\sqrt{2\cos2\theta}e^{i\theta}-1)-\frac{z}{4}\Bigg|\\
		&\geq\left|(t+s)e^{-3i\theta}+\frac{3m}{2}e^{i\theta}+(\mu-2)\frac{m}{2\sqrt{2\cos2\theta}}+(\mu+1)e^{i\theta}\right|-|\sqrt{2\cos2\theta}e^{i\theta}-1|^3\\
		&\quad{}-\left|\frac{(\mu+1)^2-p^2}{4}-2\mu-2\right||\sqrt{2\cos2\theta}e^{i\theta}-1|-\frac{|z|}{4}|\sqrt{2\cos2\theta}e^{i\theta}-1|-\frac{|z|}{4}\\
		&\geq \operatorname{Re}((t+s)e^{-3i\theta})+\frac{3m}{2}\cos\theta+(\mu-2)\frac{m}{2\sqrt{2\cos2\theta}}+(\mu+1)\cos\theta-\frac{13\sqrt{3}}{4}\\
		&\quad{}-\left|\frac{(\mu+1)^2-p^2}{4}-2\mu-2\right|\sqrt{3}-\frac{1}{4}.
		\intertext{Since for $\mu,p$ satisfying \eqref{eq f convexity}, $\mu>2$ and hence we have}
		|\psi(r,s,t;z)|&\geq \frac{15}{8\sqrt{2}}+\frac{3\mu}{2\sqrt{2}}-\left|\frac{(\mu+1)^2-p^2}{4}-2\mu-2\right|\sqrt{3}-\frac{13\sqrt{3}}{4}-\frac{1}{4}.
		\end{align*}
		Hence, we get $\psi(r,s,t;z)\neq 0$ for $r=\sqrt{2\cos2\theta}e^{i\theta},\, s=me^{3i\theta}/(2\sqrt{2\cos2\theta}),\text{ and } \operatorname{Re}((t+s)e^{-3i\theta})\geq3m^2/(8\sqrt{2\cos2\theta})$ for $m\geq n\geq 1,\, -\pi/4<\theta<\pi/4$ and $z\in\mathbb{D}$ if \[ \frac{3\mu}{2\sqrt{2}}-\sqrt{3}\left|\frac{(\mu+1)^2-p^2}{4}-2\mu-2\right|>\frac{13\sqrt{3}}{4}-\frac{15}{8\sqrt{2}}+\frac{1}{4}, \] which holds by given hypothesis. By Lemma \ref{adm lemma}, the theorem follows.
	\end{proof}

	\begin{proof}[Proof of Theorem~\ref{thm univalence of f}]
		Define the function $p:\mathbb{D}\to\mathbb{C}$ by \[ p(z)=f'_{\mu,p}(z)=\frac{h_{\mu,p}(z)}{z}.  \]
		Then the function $p(z)$ is analytic in $\mathbb{D}$ and $p(0)=1$. As the function $h_{\mu,p}(z)$ satisfies equation \eqref{de for h}, so $p(z)$ satisfies \[ z^2p''(z)+(\mu+2)zp'(z)+\left(\frac{z}{4}+\frac{(\mu+1)^2-p^2}{4}\right)p(z)-\left(\frac{(\mu+1)^2-p^2}{4}\right)=0.  \]
		Define $\psi:\mathbb{C}^3\times\mathbb{D}\to\mathbb{C}$ by \[\psi(r,s,t;z)=t+(\mu+2)s+\left(\frac{z}{4}+\frac{(\mu+1)^2-p^2}{4}\right)r-\left(\frac{(\mu+1)^2-p^2}{4}\right)=0\]
		and let $\Omega:=\{0\}.$ Then for all $z\in\mathbb{D}$ we have that $\psi(p(z),zp'(z),z^2p''(z);z)\in\Omega$. For $r,s,t$ given in Lemma \ref{adm lemma}, we have
		\begin{align*}
		\psi(r,s,t;z)&=t+s+(\mu+1)\frac{me^{3i\theta}}{2\sqrt{2\cos2\theta}}-\left(\frac{(\mu+1)^2-p^2}{4}\right)\\
		&\quad{}+\left(\frac{z}{4}+\frac{(\mu+1)^2-p^2}{4}\right)\sqrt{2\cos2\theta}e^{i\theta}\\
		&\geq \left|(t+s)e^{-3i\theta}+\frac{(\mu+1)m}{2\sqrt{2\cos2\theta}}\right|-\frac{|z|}{4}\sqrt{2\cos2\theta}\\
		&\quad{}-\left|\frac{(\mu+1)^2-p^2}{4}\right||\sqrt{2\cos2\theta}e^{i\theta}-1|\\
		&\geq \operatorname{Re}((t+s)e^{-3i\theta})+\frac{(\operatorname{Re}\mu+1)m}{2\sqrt{2\cos2\theta}}-\frac{\sqrt{2}}{4}-\left|\frac{(\mu+1)^2-p^2}{4}\right|\sqrt{3}\\
		&\geq \frac{3m^2}{8\sqrt{2\cos2\theta}}+\frac{(\operatorname{Re}\mu+1)m}{2\sqrt{2\cos2\theta}}-\frac{1}{2\sqrt{2}}-\left|\frac{(\mu+1)^2-p^2}{4}\right|\sqrt{3}.
		\intertext{According to given hypothesis, $\operatorname{Re}\mu>-1$ and $m\geq1,$ we have}
		|\psi(r,s,t;z)|&\geq \frac{(\operatorname{Re}\mu+1)}{2\sqrt{2}}-\frac{1}{8\sqrt{2}}-\left|\frac{(\mu+1)^2-p^2}{4}\right|\sqrt{3}.
		\end{align*}
		It is clear from the hypothesis that for $r=\sqrt{2\cos2\theta}e^{i\theta},s=me^{3i\theta}/(2\sqrt{2\cos2\theta})$ and $t$ such that $\operatorname{Re}((t+s)e^{-3i\theta})\geq3m^2/(8\sqrt{2\cos2\theta})$ for $m\geq n\geq1,-\pi/4<\theta<\pi/4$ and $z\in\mathbb{D},$ we have $\psi(r,s,t;z)\neq0$. Hence, by Lemma \ref{adm lemma}, the theorem follows.
	\end{proof}
	

\end{document}